\documentclass{amsart}

\usepackage{hyperref}

\newtheorem{theorem}{Theorem}[section]
\newtheorem{lemma}[theorem]{Lemma}

\theoremstyle{definition}

\newtheorem{corollary}[theorem]{Corollary}
\theoremstyle{remark}
\newtheorem{remark}[theorem]{Remark}

\numberwithin{equation}{section}

\DeclareMathOperator{\Eord}{{Eord}}

\begin{document}

\title{On the ascending chain condition for mixed difference ideals}

\author{Alexander Levin}
\address{Department of Mathematics, The Catholic University of America, Washington, D.C. 20064}

\email{levin@cua.edu}
\thanks{The author was supported by NSF Grant CCF \#1016608.}




\date\today



\begin{abstract}
We show that the ring of difference polynomials over a difference field does not satisfy the ascending chain condition for mixed difference ideals. This result solves an open problem of difference algebra stated by E. Hrushovski in connection with the development of difference algebraic geometry.
\end{abstract}

\maketitle

\section{Introduction}

Recently E. Hrushovski \cite{Hrushovski} developed the theory of difference schemes where the key role is played by mixed difference ideals and well-mixed difference rings (these rings appear as difference factor rings by mixed difference ideals; the corresponding definitions are presented in the next section). Any difference integral domain, as well as any difference ring whose structure endomorphism is Frobenius, is well-mixed. The most attractive property of the class of well-mixed rings is the fact that they admit a reasonable theory of localization; as it is shown in \cite[Chapter 3]{Hrushovski}, a difference scheme based on well-mixed rings makes sense intrinsically, without including the generating difference rings explicitly in the data.

In \cite{Hrushovski}, E. Hrushovski raised an important question whether mixed difference ideals of $R$ satisfy the ascending chain condition. In this paper, we resolve this open problem. We present a process of ``shuffling'' (see Lemma~\ref{lem:31}) for the construction of the mixed difference ideal generated by a set and use this construction to prove that a ring of difference polynomials over a difference field does not satisfy the ascending chain condition for mixed difference ideals.

\section{Background information}

Recall that a {\em difference ring} is a commutative ring $R$ considered together with a finite set $\Sigma$ of mutually commuting injective endomorphisms of $R$ called {\em translations}. If $\Sigma$ consists of one translation, we say that the difference ring $R$ is {\em ordinary}, otherwise, it is called {\em partial}. In what follows, we consider only ordinary difference rings. If $\sigma$ is the translation of such a difference ring $R$, we say that $R$ is a $\sigma$-{\em ring}. For any element $a\in R$, an element of the form $\sigma^{k}(a)$ ($k\in \mathbb{N}$) is said to be a {\em transform} of $a$. (As usual, $\mathbb{N}$ denotes the set of all nonnegative integers.) In what follows, such an element will be also denoted by $a_{k}$ (if $k=0$, we write $a$ instead of $a_{0}$). If a difference ring $R$ with translation $\sigma$ is a field, it is called a difference field or a $\sigma$-field.

Let $R$ be a difference ($\sigma$-) ring and $R_{0}$ a subring of $R$ such that $\sigma(R_{0})\subseteq R_{0}$.  Then $R_{0}$ is called a {\em difference\/} (or $\sigma$-) {\em subring\/} of $R$ and the ring $R$ is said to be a {\em difference\/} (or $\sigma$-) {\em overring\/} of $R_{0}$. In this case, the restriction of the translation on $R_{0}$ is denoted by same symbol $\sigma$. If $J$ is an ideal of $R$ such that $\sigma(J)\subseteq J$,   then $J$ is called a {\em difference\/} (or $\sigma$-) ideal of $R$. A $\sigma$-ideal $J$ of $R$ is called {\em reflexive} if for any $a\in R$, the inclusion $a_{1}\in J$ ($a\in R$) implies $a\in J$. If a prime ideal $P$ of $R$ is closed with respect to $\sigma$ (that is, $\sigma(P)\subseteq P$), it is called a {\em prime difference} (or {\em prime} $\sigma$-) {\em ideal} of $R$.

If $S$ is a subset of $R$, then the intersection of all $\sigma$-ideals of $R$ containing $S$ is denoted by $[S]$. Clearly, $[S]$ is the smallest $\sigma$-ideal of $R$ containing $S$; as an ideal, it is generated by the set  $$\{\sigma^{k}(a) | a\in S,\, k\in\mathbb{N}\}.$$  If $J = [S]$, we say that the $\sigma$-ideal $J$ is generated by the set $S$ called a {\em set of difference} (or $\sigma$-) {\em generators} of $J$. If $S$ is finite, $S=\big\{a^{(1)},\dots, a^{(m)}\big\}$, we write $$J = \big[a^{(1)},\dots, a^{(m)}\big]$$ and say that $J$ is a finitely generated $\sigma$-ideal of $R$. In this case, the elements $a^{(1)},\dots, a^{(m)}$ are said to be $\sigma$-generators of $J$.

Let $R$ be a difference ($\sigma$-) ring, $R_{0}$ a $\sigma$-subring of $R$ and $B\subseteq R$. The intersection of all $\sigma$-subrings of $R$ containing $R_{0}$ and $B$ is called the {\em $\sigma$-subring of $R$ generated by the set $B$ over $R_{0}$}, it is denoted by $R_{0}\{B\}$. As a ring, $R_{0}\{B\}$ coincides with the ring $$R_{0}[\{b_{k}\,|\,b\in B,\, k\in \mathbb{N}\}]$$ obtained by adjoining the set $\{b_{k}\,|\,b\in S,\, k\in\mathbb{N}\}$ to $R_{0}$. The set $B$ is said to be the set of {\em difference} (or $\sigma$-) {\em generators} of the difference ring $R_{0}\{B\}$ over $R_{0}$. If this set is finite, $B = \big\{b^{(1)},\dots, b^{(m)}\big\}$, we say that $R' = R_{0}\{B\}$ is a finitely generated difference (or $\sigma$-) ring extension (or overring) of $R_{0}$ and write $$R' = R_{0}\big\{b^{(1)},\dots, b^{(m)}\big\}.$$
If $I$ is a reflexive difference ideal of a difference ($\sigma$- ) ring $R$, then the factor ring $R/I$ has a unique structure of a $\sigma$-ring such that the canonical surjection $R\rightarrow R/I$ is a difference homomorphism, that is, a ring homomorphism commuting with the action of $\sigma$.  In this case, $R/I$ is said to be the {\em difference} (or $\sigma$-) {\em factor ring} of $R$ by $I$.

Let $U = \big\{u^{(\lambda)} | \lambda \in \Lambda\big\}$ be a family of elements in some $\sigma$-overring of a difference ($\sigma$-) ring $R$. We say that the family $U$ is {\em transformally} (or $\sigma$-{\em algebraically) dependent} over $R$, if the family $$U^{\sigma} = \{u^{(\lambda)}_{k}\,|\,k\in \mathbb{N},\, \lambda \in \Lambda\}$$ is algebraically dependent over $R$ (that is, there exist elements $v^{(1)},\dots, v^{(m)}\in U^{\sigma}$ and a nonzero polynomial $f$ in $m$ variables with coefficients from $R$ such that $f\big(v^{(1)},\dots, v^{(m)}\big) = 0$). Otherwise, the family $U$ is said to be {\em transformally} (or $\sigma$-{\em algebraically) independent} over $R$ or a family of {\em difference} (or $\sigma$-) {\em indeterminates} over $R$.  In the last case, the $\sigma$-ring $R\{U\}$ is called the {\em algebra of difference} (or $\sigma$-) {\em polynomials} in the difference (or $\sigma$-) indeterminates $u^{(\lambda)}$ ($\lambda \in \Lambda$) over $R$.

As it is shown in \cite[Chapter 2, Theorem I]{Cohn} (see also \cite[Proposition 3.3.7]{Levin} for the partial case), if $I$ is an arbitrary set, then there exists an algebra of difference polynomials over $R$ in a family of difference indeterminates with indices from the set $I$. If $S$ and $S'$ are two such algebras, then there exists a
difference isomorphism $S\rightarrow S'$ that leaves all elements of the ring $R$ fixed. If $R$ is an integral domain, then any algebra of difference polynomials over $R$ is an integral domain.

The algebra of difference ($\sigma$-) polynomials in a set of indeterminates $\{y_{i}\,|\,k\in\mathbb{N}, \,i\in I\}$ over $R$ is constructed as follows. Let $S$ be the polynomial $R$-algebra in a countable set of indeterminates $$\big\{y^{(i)}_{k}\,|\,i\in I, k\in\mathbb{N}\big\}.$$ For any $f\in S$, let $\sigma(f)$ denote the polynomial from $S$ obtained by replacing every indeterminate $y^{(i)}_{k}$ that appears in $f$ by $y^{(i)}_{k+1}$ and every coefficient $a\in R$ by $\sigma(a)$. We obtain an injective endomorphism $S\rightarrow S$ that extends $\sigma$ (this extension is also denoted by $\sigma$). Setting $y^{(i)}=y^{(i)}_{0}$, we obtain a $\sigma$-algebraically independent over $R$ set $\big\{y^{(i)}\: |\: i\in I\big\}$ such that $$S = R\big\{\big(y^{(i)}\big)_{i\in I}\big\}.$$ Thus, $S$ is an algebra of $\sigma$-polynomials over $R$ in a family of $\sigma$-indeterminates $\big\{y^{(i)} | i\in I\big\}$.

A difference ideal $I$ of a difference ring $R$ is called {\em perfect} if for any $a\in R$ and $k_{1},\dots ,k_{m}, d_{1},\dots, d_{m}\in \mathbb{N}$ ($m\geq 1$), the inclusion $a_{k_{1}}^{d_{1}}\dots a_{k_{m}}^{d_{m}}\in J$ implies $a\in J$. It is easy to see that every perfect ideal is reflexive and every reflexive prime difference ideal is perfect. If $B\subseteq R$ then the intersection of all perfect difference ideals of $R$ containing $B$ is the smallest perfect difference ideal containing $B$. It is denoted by $\{B\}$ and called the {\em perfect closure} of the set $B$ or the {\em perfect difference ideal generated by $B$}. (It will be always clear whether $\big\{a^{(1)},\dots, a^{(m)}\big\}$ denotes the set consisting of the elements $a^{(1)},\dots, a^{(m)}$ or the perfect difference ideal generated by these elements.)

The construction of the perfect closure via a process of ``shuffling'' is described in \cite[Chapter 3, section 2]{Cohn} where one can also find a proof of the basis theorem, which states that if an ordinary difference ring $R$ satisfies the ascending chain condition (ACC) for perfect difference ideals (in this case $R$ is called a {\em Ritt difference ring}), then the ring of difference polynomials $R\{y^{(1)},\dots, y^{(m)}\}$ in a finite set of difference indeterminates satisfies ACC for perfect difference ideals as well (the partial version of this theorem, the process of ''shuffling'' and some other properties of perfect difference ideals are presented in \cite[Section 2.3]{Levin}). The role of perfect difference ideals (in particular, in the difference version of Nullstellensatz) is similar to the role of radical ideals in commutative rings.

Every perfect difference ideal of a difference ring $R$ is the intersection of a family of reflexive prime difference ideals of $R$, and if $R$ is a Ritt difference ring, then every perfect difference ideal of $R$ is the intersection of a finite number of reflexive prime difference ideals, see \cite[Chapter 3, Section 7]{Cohn}. (Generalizations of these results to partial difference and difference-differential rings can be found in \cite[Chapter 3]{KLMP}.)

A difference ideal $I$ of an ordinary difference ($\sigma$-) ring $R$ is called {\em mixed} if the inclusion $ab\in I$ ($a, b\in R$) implies that
$ab_{1}\in I$. (Note that E. Hrushovski \cite{Hrushovski} uses the term ``well-mixed'' for a mixed difference ideal; we follow here the original terminology of R. Cohn.) One can see that every perfect ideal is mixed and that a $\sigma$-ideal $I$ is mixed if and only if for any nonempty set $S\subseteq R$,\, $$I : S = \{a\in R\,|\,as\in I\ \text{for every}\ s\in S\}$$ is a difference ideal of $R$. Furthermore, the radical of any mixed reflexive difference ideal is a perfect difference ideal. As an example of a mixed difference ideal, which is not perfect, one can consider the ideal $$J = \big[y^{2}, yy_{1}, yy_{2}\dots\big]$$ of the ring $R = K\{y\}$ of difference polynomials in one difference indeterminate $y$ over an ordinary difference field $K$. Indeed, $J$ is mixed since it consists of all difference polynomials in which every term is at least of second degree in the indeterminates $y, y_{1}, y_{2},\dots$. At the same time, $J$ is not perfect, since $y\in \{J\} = [y]$, but  $y\notin J$.

Note also that the property that determines a mixed difference ideal provides a bridge between prime difference ideals and $\sigma$-prime difference ideals in the sense of the author \cite[Definition 2.3.20]{Levin} as well as in the sense of M. Wibmer \cite[Definition 1.4.3]{Wibmer} (who calls them $\phi$-ideals denoting the basic translation by $\phi$). The last type of difference ideals plays the key role in the M. Wibmer's geometric Galois theory developed in \cite{Wibmer}. One can straightforwardly check that any mixed $\sigma$-prime (in any of the two senses) difference ideal is prime.

\section{The main result}

Let $R$ be an ordinary difference ring with translation $\sigma$. For any subset $S$ of $R$, let $\langle S\rangle$ denote the smallest mixed difference ideal containing $S$. Furthermore, let $S'$ denote the set $\{ab_{1}\,|\,ab\in S\}$ (recall that we write $b_{1}$ for $\sigma(b)$).
Note that $S\subseteq S'$, since any $s\in S$ can be written as $s1_{1}$ where $s1\in S$.

\begin{lemma}\label{lem:31}
With the above notation, let $S^{[0]} = S$ and $S^{[n]} = \big[S^{[n-1]}\big]'$ ($n=1, 2,\dots$). Then $$\langle S\rangle = \bigcup_{n=0}^{\infty}S^{[n]}\,.$$
\end{lemma}

\begin{proof} It is clear that $\bigcup_{n=0}^{\infty}S^{[n]}$ is a mixed difference ideal of $R$ containing $S$ (if an element $ab$ lies in this union, then $ab\in S^{[n]}$ for some $n$, hence $ab_{1}\in S^{[n+1]}$). Therefore, $$\langle S\rangle\subseteq\bigcup_{n=0}^{\infty}S^{[n]}.$$
Conversely, $S\subseteq \langle S\rangle$ and for any set $A\subseteq R$, the inclusion $A\subseteq \langle S\rangle$ implies that $[A]'\subseteq \langle S\rangle$, since the difference ideal $\langle S\rangle$ is mixed. Proceeding by induction on $n$, we obtain that $S^{[n]}\subseteq \langle S\rangle$ for all $n =1, 2,\dots$.
\end{proof}

\smallskip

Let $K\{y\}$ be the ring of difference polynomials in one difference indeterminate $y$ over an ordinary difference ring $K$ with translation $\sigma$. In what follows power products of the form $$y_{i_{1}}^{k_{1}}\dots y_{i_{m}}^{k_{m}},\quad m\geq 1,\ i_{\nu}, k_{\nu}\in \mathbb{N},\quad \nu = 1,\dots, m$$ will be called {\em terms} and the number $k_{1}+\dots + k_{m}$ will be called the {\em degree} of such a term. Clearly, every difference polynomial $f$ in $K\{y\}$ has a unique representation as a finite sum of terms with nonzero coefficients from $K$. If a term $t$ appears in such a representation, we will say that $t$ is contained in $f$.

Furthermore, if $t = y_{i_{1}}^{k_{1}}\dots y_{i_{m}}^{k_{m}}$ where the transforms of $y$ are arranged in the increasing order (that is, $i_{1}<\dots <i_{m}$), then the nonnegative integer $i_{m} - i_{1}$ will be called the {\em effective order} of $t$ and denoted by $\Eord(t)$. Note that for any term $t$, $$\Eord(t_{k}) = \Eord(t),\quad k\in \mathbb{N}.$$

Let $U$ denote the set of terms $\{u^{(1)}, u^{(2)},\dots\}$, where $$u^{(n)} = y_{n+2^{n}-1}y_{n+2^{n+1}-1},\quad n=0, 1, 2,\dots,$$ and let
$$A^{(1)} = u^{(0)} + u^{(1)} = yy_{1}+y_{2}y_{4},$$ $$A^{(2)} = u^{(2)} + u^{(3)} = y_{5}y_{9}+y_{10}y_{18},$$
$$\dots$$ $$A^{(n)} = u^{(2n-2)} + u^{(2n-1)} = y_{2n-3+2^{2n-2}}y_{2n-3+2^{2n-1}}+y_{2n-2+2^{2n-1}}y_{2n-2+2^{2n}},$$
$$\dots\,  .$$

\begin{remark}\label{rem:32}  Note that $$\Eord\big(u^{(n)}\big) = 2^{n},\quad n\in \mathbb{N}.$$  Furthermore, the effective orders of terms of the set $U$ are all distinct and, moreover, if $u^{(i)}, u^{(j)}, u^{(k)}$, and $u^{(l)}$ are distinct elements of $U$, then $$\Eord\big(u^{(i)}\big) - \Eord\big(u^{(j)}\big)\neq \Eord\big(u^{(k)}\big) - \Eord\big(u^{(l)}\big),$$ $$\Eord\big(u^{(i)}\big) + \Eord\big(u^{(j)}\big)\neq \Eord\big(u^{(k)}\big) + \Eord\big(u^{(l)}\big),$$ and $$\Eord\big(u^{(i)}\big) - \Eord\big(u^{(j)}\big)\neq \Eord\big(u^{(k)}\big) + \Eord\big(u^{(l)}\big).$$ This follows from the  fact that, if $a, b, c,$ and $d$ are distinct nonnegative integers, then $$2^{b} - 2^{a}\neq 2^{d} - 2^{c},\ \   2^{b} - 2^{a}\neq 2^{d} + 2^{c},\ \ \text{and}\ \ 2^{a} + 2^{b}\neq 2^{d} + 2^{c}.$$
\end{remark}

Let us fix some $m\geq 1$ and let $S$ denote the set $\big\{A^{(1)}, \dots, A^{(m)}\big\}$. Furthermore, let $W$ denote the set of all homogeneous difference polynomials of degree $2$ that lie in $[S]$ and for any $n\in \mathbb{N}$, let $W^{[n]}$ denote the set of all homogeneous difference polynomials of degree $2$ that lie in $S^{[n]}$.

\begin{lemma}\label{lem:33}
Every difference polynomial in $W$ is irreducible.
\end{lemma}

\begin{proof} Since $W\subset [S] = \big[A^{(1)}, \dots, A^{(m)}\big]$ and $A^{(1)}, \dots, A^{(m)}$ are homogeneous difference polynomials of degree $2$, an arbitrary element  $f\in W$ can be written as a linear combination of the form
\begin{equation}\label{eq:31}
f = \sum_{i=1}^{m}\sum_{j=1}^{r_{i}}\lambda^{(ij)}A_{j}^{(i)}
\end{equation}
with $0\neq\lambda^{(ij)}\in K$ ($1\leq i\leq m,\, 1\leq j\leq r_{i}$, where $r_{1},\dots, r_{m}\in \mathbb{N}$). Suppose that $f$ is reducible. Then it can be written as
\begin{equation}\label{eq:32}
f = \big(a^{(1)}y_{i_{1}} + \dots + a^{(p)}y_{i_{p}}\big)\big(b^{(1)}y_{j_{1}} + \dots + b^{(q)}y_{j_{q}}\big),
\end{equation}
where
\begin{align*}
0\neq a^{(i)}, b^{(j)}\in K,\ &(1\leq i\leq p,\, 1\leq j\leq q),\ i_{1} < \dots < i_{p},\ j_{1} < \dots < j_{q},\ \text{and}\\
&i_{r}\neq j_{s}\ \text{for any}\ r = 1,\dots, p;\, s = 1,\dots, q.
\end{align*} (The last observation follows from the fact that the effective orders of all terms of $A^{(i)}$, $1\leq i\leq m$, are positive, so the same is true for all terms in the right-hand side of~\eqref{eq:31}; therefore, $f$ cannot contain a term $y_{j}^{2} $, $j\in \mathbb{N}$.) Furthermore, $p > 1$ and $q > 1$. Indeed, if $p =1$ (or $q=1$), then $f$ is divisible by some $y_{k}$. However, this is impossible, since the effective orders of all terms that appear in the right-hand side of~\eqref{eq:31} are distinct, so both terms of any $A_{j}^{(i)}$ ($1\leq i\leq m,\, 1\leq j\leq r_{i}$) are contained in $f$. Clearly, these terms have no common factor.

It follows that the terms $$v^{(1)} = y_{i_{1}}y_{j_{1}},\, v^{(2)} = y_{i_{1}}y_{j_{2}},\, v^{(3)} = y_{i_{2}}y_{j_{1}},\ \text{and} v^{(4)} = y_{i_{2}}y_{j_{2}},$$ which appear in elements $A_{j}^{(i)}$ in the right-hand side of~\eqref{eq:31}, are transforms of four distinct elements of the set $U$ (hence they have the same effective orders as the corresponding elements of $U$) and appear in the right-hand side of~\eqref{eq:32} with nonzero coefficients. Without loss of generality we can assume that $i_{1} < j_{1}$. Then $$\Eord\big(v^{(1)}\big) = j_{1}-i_{1},\quad \Eord\big(v^{(2)}\big) = j_{2}-i_{1},$$ and either  $$\Eord\big(v^{(3)}\big) = i_{2}-j_{1},\quad \Eord\big(v^{(4)}\big) = i_{2}-j_{2}\quad (\text{if}\ i_{2} > j_{2})$$ or $$\Eord\big(v^{(3)}\big) = i_{2}-j_{1},\quad \Eord\big(v^{(4)}\big) = j_{2}-i_{2}\quad (\text{if}\ j_{1} < i_{2} < j_{2})$$ or $$\Eord\big(v^{(3)}\big) = j_{1}-i_{2},\quad \Eord\big(v^{(4)}\big) = j_{2}-i_{2}\quad (\text{if}\ i_{2} < j_{1}).$$ In the first of these three cases, one has $$\Eord\big(v^{(1)}\big) + \Eord\big(v^{(3)}\big) = \Eord\big(v^{(2)}\big) + \Eord\big(v^{(4)}\big) = i_{2} - i_{1}.$$ In the second case, $$\Eord\big(v^{(1)}\big) + \Eord\big(v^{(3)}\big) = \Eord\big(v^{(2)}\big) - \Eord\big(v^{(4)}\big) = i_{2} - i_{1}.$$ The third case gives $$\Eord\big(v^{(1)}\big) - \Eord\big(v^{(3)}\big) = \Eord\big(v^{(2)}\big) - \Eord\big(v^{(4)}\big) = i_{2} - i_{1}.$$ Since any of the last three equalities contradicts the property of effective orders of elements of $U$ established in Remark~\ref{rem:32}, we conclude that $f$ is irreducible.
\end{proof}

\begin{lemma}\label{lem:34}
With the above notation, $W = W^{[n]}$ for $n = 1, 2, \dots$.
\end{lemma}

\begin{proof} Since $[S]\subseteq S^{[n]}$, $W\subseteq W^{[n]}$. We will show the opposite inclusion by induction on $n$. If $n = 1$ and $f\in W^{[1]}\subset S^{[1]} = [S]'$, then $$f = gh_{1},$$ where $gh\in [S]$. If $g\in K$ or $h_{1}\in K$ (in the last case $h\in K$), then $f\in [S]$, that is, $f\in W$.

 If $g, h\notin K$, then $g$ and $h_{1}$ are linear homogeneous difference polynomials and hence $h$ is also such a difference polynomial.  Then $gh\in W$, which contradicts the fact that every difference polynomial in $W$ is irreducible (see Lemma~\ref{lem:33}). Thus, $W^{[1]}\subseteq W$, that is, $W^{[1]} =  W$.

Suppose that $W = W^{[n]}$ for some $n\geq 1$. Since $S$ consists of homogeneous difference polynomials of degree $2$, the sets $$[S],\quad S^{[1]} = \{fg_{1}\,|\,fg\in [S]\},\quad \big[S^{[1]}\big],\quad S^{[2]}, \dots$$ cannot contain difference polynomials with terms of degree less than $2$. Therefore, any
homogeneous difference polynomial $h$ of degree $2$ in $[S^{[n]}]$ can be written as a linear combination of the form
\begin{equation}
h = \sum_{i=1}^{d}\sum_{j=1}^{s_{d}}\mu^{(ij)}B_{j}^{(i)}
\end{equation}
with $0\neq\mu^{(ij)}\in K$,  $B^{(i)}\in W^{[n]}$ ($1\leq i\leq d,\, 1\leq j\leq s_{d}$ where $s_{1},\dots, s_{d}\in \mathbb{N}$). By the inductive hypothesis, $$W^{[n]} = W\subseteq [S].$$ Hence, $h$ is a homogeneous difference polynomial of degree $2$ in $[S]$, that is, $h\in W$. Thus, the set of all homogeneous difference polynomials of degree $2$ in $\big[S^{[n]}\big]$ coincides with $W$.

Let $f\in W^{[n+1]}$. Then $f\in \big[S^{[n]}\big]'$, so $$f = gh_{1},$$ where $gh\in \big[S^{[n]}\big]$ and, therefore, $$gh\in W.$$ Repeating the arguments of the first part of the proof we obtain that $f\in W$. Thus, $W^{[n+1]} = W$.
\end{proof}

\begin{theorem}
The ring of difference polynomials $K\{y\}$ over a on ordinary difference field $K$ does not satisfy the ascending chain condition for mixed difference ideals.
\end{theorem}

\begin{proof} Let us show that, with the above notation, the chain of mixed difference ideals
\begin{equation}\label{eq:34}
\big\langle A^{(1)}\big\rangle\subset\big\langle A^{(1)}, A^{(2)}\big\rangle\subset\dots\subset\big\langle A^{(1)}, A^{(2)},\dots A^{(m)}\big\rangle\subset\dots\end{equation}
is strictly ascending. Indeed, suppose that $$A^{(m+1)}\in\big\langle A^{(1)}, A^{(2)},\dots A^{(m)}\big\rangle$$ for some $m\geq 1$. By Lemma~\ref{lem:31}, there exists $n\in \mathbb{N}$ such that $A^{(m+1)}\in S^{[n]}$, where $S$ denotes the set $\big\{A^{(1)}, A^{(2)},\dots A^{(m)}\big\}$. Since $A^{(m+1)}$ is a homogeneous difference polynomial of degree $2$, we can apply Lemma~\ref{lem:34} and obtain that $$A^{(m+1)}\in W\subseteq [S].$$ However, the inclusion $$A^{(m+1)}\in [S] = \big[A^{(1)}, A^{(2)},\dots A^{(m)}\big]$$ is impossible, since every term of any $A^{(i)}$ ($1\leq i\leq m$) and, therefore, every term of any element of $$\big[A^{(1)}, A^{(2)},\dots A^{(m)}\big]$$ has the effective order at most $2^{2m-1}$ while $A^{(m+1)}$ contains terms of effective orders $2^{2m}$ and $2^{2m+1}$. Thus, the chain of mixed difference ideals~\eqref{eq:34} is strictly ascending.
\end{proof}

\section{Final remarks}

There are two important classes of difference ideals one of which is wider and the other one is narrower than the class of mixed difference ideals. The first one is the class of {\em complete difference ideals}. 

A difference ideal $I$ of a difference ring $R$ is called {\em complete} if the reflexive closure of its radical coincides with the perfect closure of $I$. Equivalently, one can say that a difference ideal $I$ is complete if the presence in $I$ of a product of powers of transforms of an element $a$ implies the presence in $I$ od a power of a transform of $a$.

As it is shown in \cite[Chapter 3, Section 21]{Cohn}, every mixed difference ideal is complete. Therefore, Theorem 3.5. implies the following result.

\begin{corollary}
The ring of difference polynomials $K\{y\}$ over a on ordinary difference field $K$ does not satisfy the ascending chain condition for complete  difference ideals.
\end{corollary}

The other class we would like to mention is the class of all mixed radical (not necessarily reflexive) difference ideals. The reflexive closure of any such an ideal is a perfect difference ideal. We know that perfect difference ideals satisfy ACC and we also know now that a wider class of mixed difference ideals does not satisfy ACC. At the same time the question whether mixed radical ideals satisfy the ascending chain condition is still open. 

\bibliographystyle{amsplain}

\begin{thebibliography}{10}

\bibitem{Cohn}
Cohn, R. M., Difference Algebra, {\em Interscience}, New York, 1965.

\bibitem{Hrushovski}
Hrushovski, E., The Elementary Theory of the Frobenius Automorphism. {\em Arxiv preprint. math}/0406514, 2004, pp. 1-135, arxiv.org

\bibitem{KLMP} Kondrateva, M.V.; Levin, A. B.; Mikhalev, A. V.;
Pankratev, E. V., Differential and Difference Dimension Polynomials.
{\em Kluwer Academic Publishers}, Dordrecht, 1998.

\bibitem{Levin}
Levin, A. B., Difference Algebra. {\em Springer}, New York, 2008.

\bibitem{RR}
Ritt, J. F., Raudenbush, H. W., Ideal theory and algebraic
difference equations. {\em Trans. Amer. Math. Soc.\/}, 46 (1939),
445-453.

\bibitem{Wibmer}  Wibmer, M., Geometric Difference Galois Theory, Ph.D. thesis, University of Heidelberg, 2010.

\end{thebibliography}

\end{document}